\def\Date{25.10.2012}%

\documentclass[12pt]{amsart}
\usepackage{amsmath, amsthm, latexsym, amssymb, enumerate} 
\usepackage[hidelinks]{hyperref}
\numberwithin{equation}{section} \theoremstyle{plain}
\newtheorem{thm}[equation]{Theorem}    
\newtheorem{lem}[equation]{Lemma}      \newtheorem{cor}[equation]{Corollary}
 
\theoremstyle{definition}
\newtheorem{df}[equation]{Definition}
\newtheorem{rem}[equation]{Remark}

\newtheorem{Thm}{Theorem}

      \def\N{\mathbb{N}}      \def\R{\mathbb{R}}
      \def\mino{\wedge}       \def\maxo{\vee}
                     \def\g{\gamma}
          \def\e{\varepsilon}     \def\k{\kappa}
\def\la{\lambda}                
\def\({\left(}        \def\){\right)}

\newcommand{\rr}{\mathbb R}
\newcommand{\ox}{\overline{x}}
\newcommand{\bx}{\bar{x}}
\newcommand{\bT}{\bar{T}}
\newcommand{\bu}{\bar{u}}
\newcommand{\by}{\bar{y}}
\newcommand{\bz}{\bar{z}}
\newcommand{\bv}{\bar{v}}
\newcommand{\tC}{\widetilde{C}}
\DeclareMathOperator{\len}{len}
\DeclareMathOperator{\dist}{dist}
\def\rcat(#1){rCAT$(#1)$}    
\def\hrcat(#1){rCAT$(#1;*)$} 
\def\Hrcat(#1){rCAT$(#1;H)$} 
\def\HRcat(#1;#2){rCAT$(#1;#2)$} 

\newcommand{\be}{\begin{enumerate}}
\newcommand{\ee}{\end{enumerate}}

\catcode`\@=11       
\def\rf#1{\@rf{#1}#1:;;}
\def\rfs#1{\@rfs{#1}#1:;;}
\def\rfm#1{\@rfF#1<>;;}

\def\@C{C}\def\@E{E}\def\@F{F}\def\@f{f}\def\@L{L}\def\@O{O}\def\@P{P}
\def\@Q{Q}\def\@R{R}\def\@S{S}\def\@T{T}\def\@X{X}\def\@D{D}\def\@s{s}

\def\@rf#1#2:#3;;{\def\@b{#2}
  \ifx\@b\@C Corollary~\ref{#1}\else%
  \ifx\@b\@E (\ref{#1})\else
  \ifx\@b\@F Fact~\ref{#1}\else%
  \ifx\@b\@f Figure~\ref{#1}\else%
  \ifx\@b\@L Lemma~\ref{#1}\else%
  \ifx\@b\@O Observation~\ref{#1}\else%
  \ifx\@b\@P Proposition~\ref{#1}\else%
  \ifx\@b\@Q Question~\ref{#1}\else%
  \ifx\@b\@R Remark~\ref{#1}\else%
  \ifx\@b\@S Section~\ref{#1}\else%
  \ifx\@b\@T Theorem~\ref{#1}\else%
  \ifx\@b\@X Example~\ref{#1}\else%
  \ifx\@b\@D Definition~\ref{#1}\else%
  \ifx\@b\@s \S\ref{#1}\else
  \ref{#1}\fi\fi\fi\fi\fi\fi\fi\fi\fi\fi\fi\fi\fi\fi}

\def\@rfs#1#2:#3;;{\def\@b{#2}
  \ifx\@b\@C Corollaries~\ref{#1}\else%
  \ifx\@b\@F Facts~\ref{#1}\else%
  \ifx\@b\@f Figures~\ref{#1}\else%
  \ifx\@b\@L Lemmas~\ref{#1}\else%
  \ifx\@b\@O Observations~\ref{#1}\else%
  \ifx\@b\@P Propositions~\ref{#1}\else%
  \ifx\@b\@Q Questions~\ref{#1}\else%
  \ifx\@b\@R Remarks~\ref{#1}\else%
  \ifx\@b\@S Sections~\ref{#1}\else%
  \ifx\@b\@T Theorems~\ref{#1}\else%
  \ifx\@b\@X Examples~\ref{#1}\else%
  \ifx\@b\@D Definitions~\ref{#1}\else
  \ref{#1}\fi\fi\fi\fi\fi\fi\fi\fi\fi\fi\fi\fi}

\def\@rfF<#1>#2;;{\def\@c{#2}
  \@rfs{#1}#1:;;\ifx\@c\empty\else\@rfL:#2;;\fi}

\def\@rfL:#1<#2>#3;;{\def\@b{#2}\def\@c{#3}
  #1\ifx\@b\empty\else\ref{#2}\ifx\@c\empty\else\@rfL:#3;;\fi\fi}

\catcode`\@=12       

\begin{document}
\title{The $n$-point condition and rough CAT(0)}

\author{S.M. Buckley and B. Hanson}
\address{Department of Mathematics and Statistics, NUI Maynooth, Maynooth,
Co. Kildare, Ireland}%
\email{stephen.buckley@maths.nuim.ie}
\address{Department of Mathematics, Statistics and Computer Science,
St. Olaf College, 1520 St. Olaf Avenue, Northfield, MN 55057, USA}%
\email{hansonb@stolaf.edu}

\abstract {We show that for $n\ge 5$, a length space $(X,d)$ satisfies a
rough $n$-point condition if and only if it is rough CAT(0). }
\endabstract
\date{\Date}
\maketitle


\section{Introduction} \label{S:intro}

Gromov hyperbolic spaces and CAT(0) spaces have been intensively studied;
see \cite{cdp}, \cite{gh}, \cite{va}, \cite{bh} and the references therein.
Their respective theories display some common features, notably the
canonical boundary topologies. Rough CAT(0) spaces, a class of length spaces
that properly contains both CAT(0) spaces and those Gromov hyperbolic spaces
that are length spaces, were introduced by the first author and Kurt Falk in
a pair of papers to unify as much as possible of the theories of CAT(0) and
Gromov hyperbolic spaces: the basic ``finite distance'' theory of rough
CAT(0) spaces was developed in \cite{bf1}, and the boundary theory was
developed in \cite{bf2}. As in the earlier papers, we usually write \rcat(0)
in place of rough CAT(0) below. Rough CAT(0) is closely related to the class
of bolic spaces of Kasparov and Skandalis \cite{ks1}, \cite{ks2} that was
introduced in the context of their work on the Baum-Connes and Novikov
Conjectures, and is also related to Gromov's class of CAT(-1,$\e$) spaces
\cite{gr}, \cite{dg}.

One gap in the theory developed so far is the absence of results indicating
that the class of \rcat(0) spaces are closed under reasonably general limit
processes such as pointed and unpointed Gromov-Hausdorff limits and
ultralimits. The purpose of this paper is to fill that gap.

The fact that the CAT(0) class is closed under such limit processes is a
consequence of the following well-known result (for which, see
\cite[II.1.11]{bh}):

\begin{Thm} \label{T:classical}
A complete geodesic metric space $X$ is CAT(0) if and only if it satisfies
the $4$-point condition.
\end{Thm}

In \cite[Theorem~3.18]{bf1}, it was shown that a rough variant of the
$4$-point condition is quantitatively equivalent to a weak version of
\rcat(0), and it follows that the class of weak \rcat(0) spaces is closed
under reasonably general limit processes. However it seems difficult to
decide whether or not all weak \rcat(0) spaces are necessarily \rcat(0). To
establish similar limit closure properties for \rcat(0), we prove the
following rough analogue to \rf{T:classical}; rough $n$-point conditions are
defined in \rf{S:prel}.

\begin{thm} \label{T:main}
Let $(X,d)$ be a length space. If $n \ge 5$ and $(X,d)$ satisfies a
$C$-rough $n$-point condition for some $C\ge0$, then $(X,d)$ is
$C'$-\rcat(0), where $C'=C+2\sqrt 3$. Conversely, if $(X,d)$ is
$C_0$-\rcat(0) for some $C_0>0$, then for all $n\ge 3$, $(X,d)$ satisfies a
$C$-rough $n$-point condition, where $C=(n-2)C_0$.
\end{thm}

After some preliminaries in \rf{S:prel}, we prove a pair of preparatory
lemmas in \rf{S:lemmas}. We then prove the main theorem and discuss its
limit closure consequences in \rf{S:proof}.

\section{Preliminaries} \label{S:prel}

Whenever we write $\rr^2$ in this paper, we always mean the plane with the
Euclidean metric attached. Throughout this section, $X$ is a metric space
with metric $d$ attached; any extra assumptions on $d$ will be explicitly
stated.

A \emph{$h$-short segment}, $h\ge0$, in $X$ is a path $\g:[0,L]\to X$,
$L\ge0$, satisfying
$$
\len(\g) \geq d(\g(0),\g(L)) \geq \len(\g) - h.
$$
We denote $h$-short segments connecting points $x,y \in X$ by $[x,y]_h$. It
is convenient to use $[x,y]_h$ also for the image of this path, so instead
of writing $z=\g(t)$ for some $0\le t\le L$, we often write $z\in[x,y]_h$.
Given such a path $\g$ and point $z=\g(t)$, we denote by $[x,z]_h$ and
$[z,y]_h$ respectively the subpaths $\g|_{[0,t]}$ and $\g|_{[t,L]}$,
respectively; note that both of these are $h$-short segments. A $0$-short
segment is called a {\it geodesic segment}, and we write $[x,y]$ in place of
$[x,y]_0$.

A metric space $(X,d)$ is a {\em geodesic space} if for every $x,y\in X$,
there exists at least one geodesic segment $[x,y]$. More generally, $(X,d)$
is a {\it length space} if for every $x,y\in X$ and every $h>0$, there
exists a $h$-short path $[x,y]_h$.

A \emph{$h$-short triangle} $T:=T_h(x_1,x_2,x_3)$ with vertices
$x_1,x_2,x_3\in X$ is a collection of $h$-short segments $[x_1,x_2]_h$,
$[x_2,x_3]_h$, and $[x_3,x_1]_h$ (the {\it sides} of $T$). Given such a
$h$-short triangle $T$, a \emph{comparison triangle} will mean a Euclidean
triangle $\bar{T}:=T(\bx_1,\bx_2,\bx_3)$ in $\rr^2$, such that
$|\bx_i-\bx_j|=d(x_i,x_j)$, $\;i,j\in\{1,2,3\}$. Furthermore, we say that
$\bu\in [\bx_i,\bx_j]$ is a \emph{comparison point} for $u\in [x_i,x_j]_h$,
if
$$
|\bx_i-\bu| \leq \len([x_i,u]_h)  \quad\text{and}\quad|\bu-\bx_j|
\leq \len([u,x_j]_h)\,.
$$

A {\em geodesic triangle} $T=T(x,y,z)$ is just a $0$-short triangle. Note
that in this case if $\bar{T}:=T(\bx_1,\bx_2,\bx_3)$ in $\rr^2$ is a
comparison triangle, and $\bu\in [\bx_i,\bx_j]$ is a \emph{comparison point}
for $u\in [x_i,x_j]$, then $\bu\in [\bx_i,\bx_j]$ is uniquely determined by
the equation $|\bx_i-\bu| = d(u,x_i)$.

A geodesic space $(X,d)$ is a {\it CAT(0) space} if given any geodesic
triangle $T=T(x,y,z)$ with comparison triangle $\bT=\bT(\bx,\by,\bz)$, and
any two points $u \in [x,y]$ and $v \in [x,z]$, we have $d(u,v)\le
|\bu-\bv|$, where $\bu$ and $\bv$ are comparison points for $u$ and $v$.

\begin{df}\label{D:rCAT0*}
Given $C>0$, and a function $H:X\times X\times X\to(0,\infty)$, a length
space $(X,d)$ is said to be a \emph{$C$-\Hrcat(0)} space if the following
\emph{$C$-rough CAT(0) condition} is satisfied:
\begin{equation*}
d(u,v) \leq |\bu - \bv| + C\,,
\end{equation*}
whenever
\begin{itemize}
\item $x,y,z \in X$;%
\item $T:=T_h(x,y,z)$ is a $h$-short triangle, where $h=H(x,y,z)$;%
\item $\bT:=T(\bx,\by,\bz)$ is a comparison triangle in $\rr^2$
    associated with $T$;%
\item $u,v$ lie on different sides of $T$;%
\item $\bu,\bv\in \bT$ are comparison points for $u,v$, respectively;
\end{itemize}
We call $(T_h(x,y,z),u,v)$ the {\it metric space data} and
$(T(\bx,\by,\bz),\bu,\bv)$ the {\it comparison data}.

\end{df}

\begin{df}\label{D:rCAT0}
Given $C>0$, a length space $X$ is \emph{$C$-\hrcat(0)} if there exists
$H:X\times X\times X\to(0,\infty)$ such that $X$ is $C$-\Hrcat(0). $(X,d)$
is \emph{$C$-\rcat(0)} if it is $C$-\Hrcat(0) with
\begin{equation*}
H(x,y,z)=\frac{1}{1\vee d(x,y)\vee d(x,z)\vee d(y,z)} \,.
\end{equation*}
\end{df}

Let us make some remarks about the above definitions. First, every CAT(0)
space is $C$-\rcat(0) and $C'$-\hrcat(0), with $C=2+\sqrt 3$ and $C'>0$
arbitrary; this follows from Theorem~4.5 and Corollary~4.6 of \cite{bf1}.
Trivially $C$-\rcat(0) implies $C$-\hrcat(0). Conversely, $C$-\hrcat(0)
implies $C'$-\rcat(0) for $C':=3C+2+\sqrt 3$; see \cite[Corollary~4.4]{bf1}.

The explicit $H$ in the \rcat(0) condition has proved to be useful, but one
situation where \hrcat(0) is needed is when the parameter $C$ is close to
$0$. In particular, we show in \rf{T:rough 5-pt} that if $(X_n)$ is a
sequence of $C_n$-\hrcat(0) spaces with $C_n\to 0$, then under rather
general conditions the resulting limit space is necessarily CAT(0). {\it A
fortiori}, we could change the $C_n$-\hrcat(0) hypothesis in this result to
$C_n$-\rcat(0), but such a variant is of no real interest since a length
space satisfying a $C$-\rcat(0) condition for $C<1/2$ has diameter at most
$C$ (as a hint, in a space of diameter larger than this, consider a triangle
$T(x,x,x)$ containing a side $[x,x]$ that moves away from $x$ and back
again). In particular, only a one-point space can be $C$-\rcat(0) for all
$C>0$. By contrast, the class of spaces that are $C$-\hrcat(0) for all $C>0$
is quite large: it includes, for instance, all CAT(0) spaces (as mentioned
above), as well as the deleted Euclidean plane.

We now introduce the concept of $C$-rough subembeddings (into $\rr^2$),
which we use to define rough $n$-point conditions.

\begin{df} \label{c-rough-subembedding}
Let $(X,d)$ be a metric space, $C \ge 0$ and $n \ge 3$ be an integer.
Suppose $x_i \in X$ and $\ox_i\in \rr^2$ for $0\le i \le n$, with $x_0=x_n$
and $\ox_0=\ox_n$. We say that $(\ox_1,\ox_2,\dots,\ox_n)$ is a {\em
$C$-rough subembedding} of $(x_1,x_2,\dots,x_n)$ into $\rr^2$ if
\begin{align*}
d(x_i,x_{i-1}) \;&=\;   |\ox_i-\ox_{i-1}|\,, &&1 \le i \le n\,, \\
d(x_1,x_i)     \;&\le\; |\ox_1-\ox_i|\,,     &&2 \le i \le n\,, \qquad
  \text{and} \\
d(x_i,x_j)     \;&\le\; |\ox_i-\ox_j|+C\,,   &&2 \le i,j \le n\,.
\end{align*}
\end{df}

\begin{df} \label{n-point-condition}
Let $n\ge 3$ be an integer.  A metric space $(X,d)$ satisfies the {\em
$C$-rough $n$-point condition}, where $C \ge 0$, if every $n$-tuple in $X$
has a $C$-rough subembedding into $\rr^2$. We say that $X$ satisfies a rough
$n$-point condition if it satisfies a $C$-rough $n$-point condition for some
$C$. The {\it $n$-point condition} is the $0$-rough $n$-point condition.
\end{df}

We note that our notion of a rough $5$-point condition is somewhat analogous
to the {\it mesoscopic curvature} notion of Delzant and Gromov \cite{dg}
which they call CAT${}_\e(\k)$, although that paper is concerned with
$\k<0$, whereas our notion corresponds to $\k=0$.

Before proceeding further, let us discuss these conditions. If we vary just
one of the parameters $C$ and $n$ in the $C$-rough $n$-point condition, it
is easy to see that decreasing $C$ or increasing $n$ gives a stronger
condition; note that to deduce the $C$-rough $(n-1)$-point condition from
the $C$-rough $n$-point condition, we simply take $x_n=x_{n-1}$. The
$3$-point condition is satisfied by all metric spaces.

For geodesic spaces, the $4$-point condition is equivalent to CAT(0); see
\cite[II.1.11]{bh}. For length spaces, a $C$-rough $4$-point condition is
quantitatively equivalent to a weaker version of \rcat(0) in which the
$C$-rough CAT(0) condition is assumed for metric space data
$(T_h(x,y,z),u,v)$ only when $v$ is one of the vertices $x,y,z$; see
\cite[Theorem~3.18]{bf1}. However it seems difficult to decide whether or
not weak \rcat(0) spaces are necessarily \rcat(0). We do not address that
issue in this paper, but we will show that, among length spaces, \rcat(0) is
quantitatively equivalent to a rough $n$-point condition for $n>4$. Thus the
class of weak \rcat(0) spaces coincides with the class of length spaces
satisfying a rough $4$-point condition, and the class of \rcat(0) spaces
coincides with the class of length spaces satisfying an $n$-point condition
for any value (or all values) of $n>4$, but we cannot say whether or not a
rough $4$-point condition implies a rough $n$-point condition for $n>4$.

\section{Two lemmas} \label{S:lemmas}

The proof of \rf{T:main} requires the following two lemmas. The first is a
restatement of \cite[Lemma 3.12]{bf1}.

\begin{lem}\label{r-uniq-geo}
Let $x,y$ be a pair of points in the Euclidean plane $\rr^2$, with
$l:=|x-y|>0$. Fixing $h>0$, and writing $L:=l+h$, let $\g:[0,L]\to \rr^2$ be
a $h$-short segment from $x$ to $y$, parameterized by arclength. Then there
exists a map $\la:[0,L]\to[x,y]$ such that $\la(0)=x$, $\la(L)=y$, and
\begin{alignat*}{2}
|\la(t)-x|&\le |\g(t)-x|\,, \qquad & 0\le t\le L\,, \\
|\la(t)-y|&\le |\g(t)-y|\,, \qquad & 0\le t\le L\,, \\
\delta(t):=\dist(\g(t),\la(t)) &\le M := \frac{1}{2}\sqrt{2lh+h^2}\,,
  \qquad & 0\le t\le L\,.
\end{alignat*}
In particular if $h=\e/(1\maxo l)$ for some $0<\e\le 1$, then $\delta(t)\le
\sqrt{3\e}/2$ for all $0\le t\le L$.
\end{lem}

\begin{lem} \label{triangle-comparison}
Assume $x_i,x'_i \in \rr^2$ for $i=0,1,2$, with $u_i \in [x_0,x_i]$ and
$u'_i \in [x'_0,x'_i]$ for $i=1,2$ and let %
\begin{equation*}
h=\frac{\e}{1 \vee |x_0'-x_1'| \vee |x_0'-x_2'|}\,,
\end{equation*}
for some $0<\e\le 1$. Suppose further that %
\begin{alignat*}{2}
|x_1-x_2|                  &\;=\;   |x'_1-x'_2|\,, & \\
|x'_0-x'_i| \le |x_0-x_i|  &\;\le\; |x'_0-x'_i|+h, \qquad &i=1,2\,. \\
\intertext{and} %
\frac{|u_i-x_0|}{|x_0-x_i|} &\;=\;  \frac{|u'_i-x'_0|}{|x'_0-x'_i|}\,,
  \qquad &i=1,2\,. %
\end{alignat*} %
Then $|u_1-u_2| \le |u'_1-u'_2|+\sqrt{3\e}\,.$
\end{lem}

\begin{proof}
Set %
\begin{equation*}
s=\frac{|u_1-x_0|}{|x_1-x_0|}=\frac{|u_1'-x_1'|}{|x_1'-x_0'|}
\end{equation*}
and %
\begin{equation*}
t=\frac{|u_2-x_0|}{|x_2-x_0|}=\frac{|u_2'-x_0'|}{|x_2'-x_0'|}.
\end{equation*} %
We assume without loss of generality that $s \le t$. An elementary
calculation using the parallelogram law shows that given $x,y,z$ in the
Euclidean plane with $w \in [y,z]$ and $|w-y|=r|z-y|$ we have
\begin{equation} \label{parallelogram}
|x-w|^2=(1-r)|x-y|^2+r|x-z|^2-r(1-r)|y-z|^2\,.
\end{equation}
%
%
Using (\ref{parallelogram}) twice, we get %
\begin{equation} \label{u_1-u_2}
|u_1-u_2|^2=st|x_1-x_2|^2+t^2\(1-\frac{s}t\)|x_0-x_2|^2-st\(1-\frac{s}t\)|x_0-x_1|^2
\end{equation} and similarly%
\begin{equation} \label{u_1'-u_2'} %
|u_1'-u_2'|^2 =
  st|x_1'-x_2'|^2+t^2\(1-\frac{s}t\)|x_0'-x_2'|^2-st\(1-\frac{s}t\)|x_0'-x_1'|^2.
\end{equation}

Setting $|u_1-u_2|=|u_1'-u_2'|+d$ and subtracting (\ref{u_1'-u_2'}) from
(\ref{u_1-u_2}), we get
\begin{align*}
2d|u_1'-u_2'|+d^2 &= t^2\(1-\frac{s}t\)
  \(|x_0-x_2|^2-|x_0'-x_2'|^2\)-\\
&-st\(1-\frac{s}t\)\(|x_0-x_1|^2-|x_0'-x_1'|^2\) \\
&\le t^2\(1-\frac{s}t\)\(|x_0-x_2|^2-|x_0'-x_2'|^2\)\\
&\le t^2\(1-\frac{s}t\)\(2h|x_0'-x_2'|+h^2\) \le 3\e.
\end{align*}
In particular $d\le\sqrt{3\e}$, as required.
\end{proof}

\section{Proof and consequences} \label{S:proof}

Here we prove \rf{T:main} and discuss some consequences. First we need a
definition.

\begin{df} \label{gluing}
Suppose $(S,d_S)$ is a metric space, and that for $i=1,2$, we have a metric
space $(X_i,d_i)$, a closed subspace $S_i\subset X_i$, and a surjective
isometry $f_i:S\rightarrow S_i$. We then define the gluing of $X_1$ and
$X_2$ along $S_1,S_2$ (denoted by $X=X_1 \sqcup_S X_2$) as the quotient of
the disjoint union of $X_1$ and $X_2$ under the identification of $f_1(s)$
with $f_2(s)$ for each $s \in S$. The glued metric $d$ on $X$ is defined by
the equations $d|_{X_i\times X_i}=d_i, i=1,2$ and
$$
d(x_1,x_2) = \inf_{s \in S}\,(d_1(x_1,f_1(s))+d_2(f_2(s),x_2))\,, \qquad
  x_1 \in X_1,\; x_2 \in X_2.
$$
\end{df}

We note the following easily verified facts about $(X,d):=X_1 \sqcup_S X_2$
defined by gluing as above:
\begin{itemize}
\item $d$ restricted to $X_i$, $i=1,2$, coincides with $d_i$;%
\item every geodesic segment in $X_i$, $i=1,2$, is also a geodesic
    segment in $X$.
\end{itemize}

We now prove the following slight improvement of \rf{T:main}.

\begin{thm} \label{T:main1}
Let $(X,d)$ be a length space. If $n\ge 5$ and $(X,d)$ satisfies a $C$-rough
$n$-point condition for some $C\ge 0$, then $(X,d)$ is $C'$-\rcat(0) and
$C''$-\hrcat(0), where $C'=C+2\sqrt 3$ and $C''>C$ is arbitrary. Conversely,
if $(X,d)$ is $C_0$-\hrcat(0) for some $C_0>0$, then for all $n \ge 3$,
$(X,d)$ satisfies a $C$-rough $n$-point condition, where $C=(n-2)C_0$.
\end{thm}

\begin{proof}
Assume that $(X,d)$ is a length space.  We first prove the forward
implication, so we assume that $n\ge 5$ and that $(X,d)$ satisfies a
$C$-rough $n$-point condition for some $C \ge 0$.  It follows trivially that
$(X,d)$ satisfies a $C$-rough 5-point condition.  Let $T:=T_h(x,y,z)$ be a
$h$-short geodesic triangle in $X$, where
\begin{equation}\label{E:h bounde}
h = H(x,y,z):= \frac{\e}{1\vee d(x,y)\vee d(x,z)\vee d(y,z)} \, ,
\end{equation}
and $0<\e\le 1$ is fixed but arbitrary. Assume also that $u\in [x,y]_h$ and
$v\in [x,z]_h$. Let $(x',u',y',z',v')$ be a $C$-rough subembedding of
$(x_1,x_2,x_3,x_4,x_5)=(x,u,y,z,v)$ into $\rr^2$, so in particular we have
$$ d(x,y)\le|x'-y'|\,,\qquad d(x,z)\le|x'-z'|\,, \qquad d(y,z)=|y'-z'|\,, $$
and
\begin{equation}
d(u,v) \;\le\; |u'-v'|+C. \label{ineq:uv}
\end{equation}
From the definition of a $C$-rough subembedding and the fact that $T$ is
$h$-short, it follows that the piecewise linear paths $\g_1=[x',u']\cup
[u',y']$ and $\g_2=[x',v']\cup [v',z']$ are both $h$-short.  Thus, by Lemma
\ref{r-uniq-geo} we can choose $u''\in [x',y']$ and $v''\in [x',z']$ such
that
\begin{equation} \label{u'v'}
|u'-u''|\le \frac{\sqrt{3\e}}2 \quad \text{ and } \quad |v'-v''|\le
  \frac{\sqrt{3\e}}2
\end{equation}
and such that
\begin{equation} \label{u''}
|u''-x'| \le |u'-x'| \quad \text{ and } \quad |u''-y'| \le |u'-y'|
\end{equation}
and
\begin{equation} \label{v''}
|v''-x'| \le |v'-x'| \quad \text{ and } \quad |v''-z'| \le |v'-z'|.
\end{equation}

Now let $\bar{T}=T(\bx,\by,\bz)$ be a comparison triangle for $T$ and choose
$\bu \in [\bx,\by], \bv \in [\bx,\bz]$ satisfying:
\begin{equation} \label{uv-ratio}
\frac{|\bu-\bx|}{|\bx-\by|}=\frac{|u''-x'|}{|x'-y'|} \quad\text{ and }\quad
\frac{|\bv-\bx|}{|\bx-\bz|}=\frac{|v''-x'|}{|x'-z'|}.
\end{equation}
Since $|\bx-\by|=d(x,y)\le |x'-y'|$, it follows from (\ref{u''}) and
(\ref{uv-ratio}) that
$$ |\bu-\bx| \le |u''-x'| \le |u'-x'| $$
and
$$ |\bu-\by| \le |u''-y'| \le |u'-y'|\,, $$
so $\bu$ is a comparison point for $u$. Similarly $\bv$ is a comparison
point for $v$.  Finally, using (\ref{ineq:uv}) and (\ref{u'v'}), we see
that
$$ d(u,v) \le |u'-v'|+C \le |u''-v''|+C+\sqrt{3\e}\,, $$
and so by Lemma \ref{triangle-comparison}, we get
$$ d(u,v) \le |\bu-\bv|+C+2\sqrt{3\e}\,. $$
Thus $(X,d)$ is $C'$-\hrcat(0), with $C'=C+2\sqrt{3\e}$. Taking $\e=1$, we
see that $X$ is $C'$-\rcat(0), where $C'=C+2\sqrt 3$. Letting $\e>0$ be
sufficiently small, we see that $X$ is $C''$-\hrcat(0).

We next proceed with the reverse implication, so let us assume that $(X,d)$
is $C'$-\hrcat(0). We will prove that $(X,d)$ satisfies the $C_n$-rough
$n$-point condition, where $C_n:=(n-2)C'$ and $n\ge 3$.

The proof will involve induction, but using a stronger inductive hypothesis
which involves not just a set of $n$ points, but an $n$-gon with these
points as vertices. Additionally, the inductive process requires us to
establish simultaneously a CAT(0) version of the result. Note that it
suffices to prove the result for sets of distinct points, since the desired
conditions for $n$ points with at least one repeated point follows
immediately from the condition for $n-1$ points.

Given $u_1,u_2,\dots,u_n \in X$, $n\ge 3$, we say that $P$ is a $h$-short
$n$-gon (with vertices $u_1,u_2,\dots,u_n=u_0$) if $P$ is the union of
$h$-short paths $[u_{i-1},u_i]_h$ for $i=1,2,\dots,n$. An $n$-gon is {\it
geodesic} if it is $0$-short. We say that $h$ is {\it suitably small} if
$h<H(u_i,u_j,u_k)$ for all $1\le i,j,k\le n$.

Suppose
\begin{itemize}
\item $Q$ is a geodesic $n$-gon with distinct vertices $(v_i)_{i=1}^n$
    and associated metric $d'$;%
\item $P$ a $h$-short $n$-gon with distinct vertices $(u_i)_{i=1}^n$ and
    associated metric $d$;%
\item $F:Q\to P$ is a map with $F(v_i)=u_i$, $1\le i\le n$.
\end{itemize}
Since a geodesic segment is isometrically equivalent to a segment on $\R$,
we can view the restriction of $F$ to a single side of $Q$ as being a path,
and hence define the path length $\len(F;x,y)$ to be the length of the
associated path segment from $F(x)$ to $F(y)$. We call $F:Q\to P$ a {\it
constant speed $n$-gon map} if $P,Q,F$ are as above, and if for each $1\le
i\le n$ there is a constant $K_i$ such that $\len(F;x,y)=K_id'(x,y)$
whenever $x,y\in [v_{i-1},v_i]$. It is easy to see that, given any $P,Q$ as
above, a constant speed $n$-gon map always exists.

Given the following data:
\begin{itemize}
\item a $h$-short $n$-gon $P$ with distinct vertices $u_1,u_2,\dots,u_n
    \in X$, where $(X,d)$ is a metric space and $h$ is suitably small;
    \\[-6pt]
\item a constant speed $n$-gon map $F:Q\to P$, where $Q$ is a geodesic
    $n$-gon with distinct vertices $v_1,v_2,\dots,v_n\in Y$, and
    $(Y,d')$ is a CAT(0) space,
\end{itemize}
we define a hypothesis $A_n(P,h;F,Q,d',C_n)$:
\begin{alignat}{3}
u_i            &\;=\;   F(v_i)\,,          &&\quad 1\le i\le n\,,\notag\\
d(u_{i-1},u_i) &\;=\;   d'(v_{i-1},v_i)\,, &&\quad 1\le i\le n\,,\notag\\
d(u_1,u_i)     &\;\le\; d'(v_1,v_i)\,,     &&\quad 2\le i\le n\,,\notag\\
\len([F(x),u_i]_h) &\;\ge\; d'(x,v_i)\,,
  &&\quad x \in Q,\;\; v_i \text{ a vertex adjacent to } x\,,
  \label{E:vertex-length}\\
d(F(x),F(y)) &\;\le\; d'(x,y)+C_n\,,       &&\quad x,y\in Q\,.
 \label{E:d(x,y)}
\end{alignat}
{\it The inductive hypothesis for $n$} is that for all $P,h$ as above, there
exist data $(F,Q,d')$ such that $A_n(P,h;F,Q,d',C_n)$ holds, and such that
$(Q,d')$ is a convex Euclidean $n$-gon in $\rr^2$ with $d'$ being the
Euclidean metric. This implies the desired $C_n$-rough $n$-point embedding:
the vertices of $Q$ give the rough subembedding of the vertices of $P$. We
have defined the hypothesis $A_n(P,h;F,Q,d',C_n)$ in the more general
context of a CAT(0) space $Y$ because we will need this along the way.

The CAT(0) version of our inductive hypothesis for $n$ is that for all
geodesic $n$-gons $P$ as above, there exist data $(F,Q,d')$ such that
$A_n(P,0;F,Q,d',0)$ holds, and such that $(Q,d')$ is a convex Euclidean
$n$-gon in $\rr^2$ with $d'$ being the Euclidean metric. Note also that with
$h=0$ and $C_n=0$ we get equality in \rf{E:vertex-length}, and \rf{E:d(x,y)} simplifies
to
\begin{equation} \label{E:d(x,y) CAT}
d(F(x),F(y)) \le |x-y|\,.
\end{equation}

It is a routine task to use the $C'$-\rcat(0) condition to verify the
inductive hypothesis for $n=3$ (and CAT(0) to verify the CAT(0) variant of
the inductive hypothesis for $n=3$), so assume that it holds for $n=k\ge 3$.
Let $P$ be a given $h$-short $(k+1)$-gon, where $h$ is sufficiently
small. We draw a $h$-short path from $u_1$ to $u_k$ that splits $P$ into a
$h$-short $k$-gon $P_1$ with vertices $u_1,\dots, u_k$, and a $h$-short
triangle $P_2$ with vertices $u_1, u_k, u_{k+1}$. Let $F_i: Q_i \rightarrow
P_i$, $i=1,2$ be the maps guaranteed by our inductive hypothesis for $n=k$
and the easy case $n=3$, where $Q_1$ is a convex $k$-gon with vertices
$v_1,v_2,\dots,v_k \in \rr^2$ and $Q_2$ is a triangle with vertices
$v_1,v_k,v_{k+1}$. By use of isometries of $\rr^2$, we may assume that the
sides from $v_1$ to $v_k$ in $Q_1$ and in $Q_2$ are the same, and that $Q_1$
and $Q_2$ are on opposite sides of this line segment (so the interiors of
$Q_1$ and $Q_2$ are disjoint).

We now let $(Q,d')$ be the metric space formed by gluing $Q_1$ and $Q_2$
together along $S=[v_1,v_k]$, so $Q'=Q_1 \sqcup_S Q_2$.  Let $Q$ be the
$(k+1)$-gon with vertices $v_1,v_2,\dots,v_{k+1}$ and define $F:P\rightarrow
Q$ by
\begin{displaymath}
   F(x) = \begin{cases}
       F_1(x)\,, &x \in Q_1 \cap Q\,,\\
       F_2(x)\,, &x \in Q_2 \cap Q\,.
     \end{cases}
\end{displaymath}

Note that the fact that each $F_i$ is a constant speed map ensures that $F$ is well-defined.

We wish to prove $A_{k+1}(P,h;F,Q,d',C_{k+1})$. In view of the construction,
it suffices to verify \rf{E:d(x,y)}, and for this we may assume that $x \in
Q_1$ and $y \in Q_2$.  Let $\gamma$ be the geodesic in $Q$ connecting $x$ to $y$.  It follows that
$\gamma=[x,v] \cup [v,y]$, where $v \in [v_1,v_k]$.

Using \rf{E:d(x,y)} for $P_1$ and $P_2$ and the definition of the gluing metric $d'$ on $Q$, we thus get
\begin{align*}
d(F(x),F(y)) &\le d(F(x),F(v))+d(F(v),F(y)) \\
&\le |x-v|+C_k+|v-y|+C_3 \\
&=   d'(x,y)+C_{k+1}.
\end{align*}

Essentially the same argument allows us to deduce $A_{k+1}(P,0;F,Q,d',0)$
from the CAT(0) version of our inductive hypothesis.

If $Q$ happens to be convex, we are done with the proof so assume that $Q$ is not convex.  Then
the interior angle at either $v_1$ or $v_k$ exceeds $\pi$.  Assume without loss of generality that
the interior angle at $v_1$ is larger than $\pi$.  Now the union of the two geodesic segments
$[v_{k+1},v_1]$ and $[v_1,v_2]$ is also a geodesic segment, and so by
eliminating $v_1$ as a vertex, we may consider $Q$ to be a geodesic $k$-gon
with vertices $v_2,v_3,\dots,v_{k+1}$. We also note that $Q$ is CAT(0) since
$Q_1$ and $Q_2$ are CAT(0) and the gluing set $[v_1,v_k]$ is convex; see
\cite[II.11.1]{bh}.  Applying the CAT(0) version of our induction assumption
to $Q$, we get a map $G: R \rightarrow Q$, where $R$ is a convex $k$-gon in
$\rr^2$ with vertices $w_2,w_3,\dots,w_{k+1}$ satisfying:
\begin{alignat*}{3}
v_i           &\;=\;    G(w_i)\,,        \qquad &&2\le i\le k+1\,,  \\
|v_{i-1}-v_i| &\;=\;    |w_{i-1}-w_i|\,, \qquad &&3\le i\le k+1\,, \\
|v_2-v_{k+1}| &\;\le\;  |w_2-w_{k+1}|\,, && \\
|G(y)-v_i|    &\;=\;    |y-w_i|\,, \qquad
  &&\text{whenever } y \in R,\;\; w_i \text{ a vertex adjacent to } y\,, \\
|G(y)-G(z)|   &\;\le\;  |y-z|\,, \qquad &&y,z\in Q\,.
\end{alignat*}
We now view $R$ as a convex $(k+1)$-gon by identifying $G^{-1}(v_1)$ as an
extra vertex (with interior angle $\pi$). Then $F \circ G$ is the desired
mapping (for both the \rcat(0) and CAT(0) variants of our inductive
hypothesis). Thus we have established the inductive hypothesis for $n=k+1$
and we are done with the proof.
\end{proof}

For completeness we state a CAT(0) variant of \rf{T:main}.

\begin{thm} \label{T:main0}
A complete geodesic space $(X,d)$ satisfies the $n$-point condition for
fixed $n\ge 4$ if and only if it is CAT(0).
\end{thm}

\begin{proof}
Since \rf{T:classical} already tells us that the $4$-point condition is
equivalent to CAT(0), it suffices to prove that CAT(0) implies the $n$-point
condition for each $n>4$. But this follows from the CAT(0) version of our
inductive hypothesis which was established in the proof for all $n\in\N$.
\end{proof}

\begin{rem}
By examining the above proof, we see that if $X$ is $C$-\hrcat(0), then $X$
is $C'$-\HRcat(0;H') with $C'=3C+2\sqrt{3\e}$, $0<\e\le 1$, and
\begin{equation}\label{E:He}
H'(x,y,z):= \frac{\e}{1\vee d(x,y)\vee d(x,z)\vee d(y,z)} \,.
\end{equation}
Taking $\e=1$, this slightly strengthens \cite[Corollary~4.4]{bf1} which
states that $C$-\hrcat(0) implies $C'$-\rcat(0) for $C':=3C+2+\sqrt 3$. Also
interesting is the case $\e=1\mino (C^2/3)$: this shows that the
$C$-\Hrcat(0) condition with arbitrary $H$ implies the $(5C)$-\HRcat(0;H')
condition with the explicit $H'$ given by \rf{E:He}.
\end{rem}

As mentioned in the Introduction, CAT(0) is preserved by various limit
operations, including pointed Gromov-Hausdorff limits and ultralimits
\cite[II.3.10]{bh}. The trick is to use the $4$-point condition and the
concept of a $4$-point limit. A very similar argument, with the $4$-point
condition replaced by our rough $5$-point condition, will give us similar
results for \rcat(0) spaces. We begin with a definition of $n$-point limits.

\begin{df}
A metric space $(X,d)$ is an $n$-point limit of a sequence of metric spaces
$(X_m,d_m)$, $m\in\N$, if for every $\{x_i\}_{i=1}^n\subset X$, and $\e>0$,
there exist infinitely many integers $m$ and points $x_i(m)\in X_m$, $1\le
i\le n$, such that $|d(x_i,x_j)-d_m(x_i(m),x_j(m))|<\e$ for $1\le i,j\le n$.
\end{df}

We are now ready to state a $5$-point limit result. Note that, since any
$n$-point limit of $((X_m,d_m))_{m=1}^\infty$ $(X,d)$ is also $n'$-point
limit of this sequence of spaces for all $n'\le n$, the following result
also holds if $5$ is replaced by any larger integer. The proof of this
result, which is very similar to the corresponding result for CAT(0) and
4-point limits given in \cite[II.3.9]{bh}, is included for completeness.

\begin{thm}\label{T:rough 5-pt}
Suppose the length space $(X,d)$ is a $5$-point limit of $(X_m,d_m)$,
$m\in\N$, where $(X_m,d_m)$ is $C_m$-\hrcat(0) for some constant $C_m$. If
$C_m\le C$ for all $m\in\N$, then $(X,d)$ is a $\tC$-\rcat(0) space, where
$\tC=3C+2\sqrt 3$. If $C_m\to 0$, and $(X,d)$ is complete, then $(X,d)$ is a
CAT(0) space.
\end{thm}

\begin{proof}
Suppose first that $C_m\le C$ for all $m\in\N$. Let $(x_i)_{i=1}^5$ be an
arbitrary $5$-tuple of points in $(X,d)$, and suppose that it is the
$5$-point limit of the $5$-tuples $(x_i(m))_{i=1}^5$ in $X_m$.  By passing
to a subsequence if necessary, we may assume that $d(x_i(m),x_j(m))\to
d(x_i,x_j)$ for all $1\le i,j\le 5$.

By \rf{T:main1}, every $(X_m,d_m)$ satisfies a $C'$-rough $5$-point
condition, where $C':=3C$, so there exists a $C'$-rough subembedding
$(\ox_1(m),\ox_2(m),\dots,\ox_5(m))$ of $(x_1(m),x_2(m),\dots,x_5(m))$ into
$\rr^2$, for each $m\in\N$. Since translation is an isometry in $\rr^2$, we
may assume that the points $\ox_1(m)$ coincide for all $m\in\N$. Thus all
$5$-tuples are contained in a disk of finite radius and by passing to a
subsequence if necessary we may assume that $\ox_i(m)$ converges to some
point $\ox_i$ as $m\to\infty$, for all $1\le i\le m$. It follows readily
that $(\ox_i)_{i=1}^5$ is a $C'$-rough subembedding of $(x_i)_{i=1}^5$ in
$\rr^2$. Thus $(X,d)$ satisfies the $C'$-rough $5$-point condition. By again
using \rf{T:main1}, we deduce that $(X,d)$ is a $\tC$-\rcat(0) space where
$\tC=C'+2\sqrt 3$.

If in fact $C_m\to 0$, then $(X_m,d_m)$ satisfies a $(3C_m)$-rough $5$-point
condition, and it follows as above that $(X,d)$ satisfies the $0$-rough
$5$-point condition, and hence the $4$-point condition. This together with
completeness and approximate midpoints (as follows from the fact that
$(X,d)$ is a length space) implies that $(X,d)$ is a CAT(0) space: see
\cite[II.1.11]{bh}.
\end{proof}

With \rf{T:rough 5-pt} in hand, it is now routine to deduce the following
corollary.

\begin{cor}\label{C:limit}
Suppose $(X,d)$ is a length space and suppose $(X_m,d_m)$, $m\in\N$, form a
sequence of $C$-\rcat(0) spaces. Writing $\tC=3C+2\sqrt 3$, the following
results hold.
\begin{enumerate}[(a)]
\item If $(X,d)$ is a (pointed or unpointed) Gromov-Hausdorff limit of
    $(X_m,d_m)$ then $(X,d)$ is a $\tC$-\rcat(0) space.
\item If $(X,d)$ is an ultralimit of $(X_m,d_m)$, then $(X,d)$ is a
    $\tC$-\rcat(0) space.
\item If $X$ is \rcat(0), then the asymptotic cone $\text{Cone}_\omega
    X:=\lim_\omega (X,d/m)$ is a CAT(0) space for every non-principal
    ultrafilter $\omega$.
\end{enumerate}
\end{cor}

Note that in the proof of \rf{C:limit}(c), we need the fact that
$\text{Cone}_\omega X$ is complete, but this is true because ultralimits are
always complete \cite[I.5.53]{bh}.

In each part of \rf{C:limit}, the existence of an approximate midpoint for
arbitrary $x,y\in X$ (meaning a point $z$ such that $d(x,z)\maxo d(y,z)\le
\e+d(x,y)/2$ for fixed but arbitrary $\e>0$) follows easily from the
hypotheses, and so $(X,d)$ is easily seen to be a length space if it is
complete. Thus \rf{C:limit} generalizes the $\k=0$ case of
\cite[II.3.10]{bh}(1), (2), where the spaces are assumed to be CAT(0) rather
than \rcat(0) and the limit space $(X,d)$ is assumed to be complete rather
than a length space.


\end{document}